\documentclass[copyright]{eptcs}
\providecommand{\event}{GandALF 2013} 
\usepackage{breakurl}             
\usepackage{epsfig}
\usepackage{amsmath}
\usepackage{amssymb}

\newtheorem{theo}{Theorem}

\newtheorem{defin}[theo]{Definition}
\newtheorem{lemma}[theo]{Lemma}
\newtheorem{propo}[theo]{Proposition}
\newtheorem{coro}[theo]{Corollary}

\newenvironment{proof}{\emph{Proof:}$\\$}{$\\\Box\\$}

\newcommand\dom{\texttt{Dom}}
\newcommand\tuple{\vec}
\newcommand{\indep}[2]{#1 ~\bot~ #2}
\newcommand{\indepc}[3]{#1 ~\bot_{#3}~ #2}
\newcommand{\nindepc}[3]{#1 ~\not \!\!\bot_{#3}~ #2}

\newcommand {\parts}{\mathcal P}

\newcommand \all{\texttt{All}}
\newcommand \nonempty{\texttt{NE}}
\newcommand{\ESO}{\text{ESO}}
\newcommand{\FO}{\text{FO}}
\newcommand{\D}{\mathcal D}
\newcommand{\height}{\texttt{ht}}

\title{Upwards Closed Dependencies in Team Semantics}
\author{Pietro Galliani\thanks{Research supported by Grant 264917 of the Academy of Finland.}
\institute{Department of Mathematics and Statistics\\Helsinki, Finland}
\email{pgallian@gmail.com}}
\def\titlerunning{Upwards Closed Dependencies}
\def\authorrunning{P. Galliani}
\begin{document}
\maketitle

\begin{abstract}
We prove that adding upwards closed first-order dependency atoms to first-order logic with team semantics does not increase its expressive power (with respect to sentences), and that the same remains true if we also add constancy atoms. As a consequence, the negations of functional dependence, conditional independence, inclusion and exclusion atoms can all be added to first-order logic without increasing its expressive power.

Furthermore, we define a class of bounded upwards closed dependencies and we prove that unbounded dependencies cannot be defined in terms of bounded ones. 
\end{abstract}

\section{Introduction}
Team semantics is a generalization of Tarski's semantics in which formulas are satisfied or not satisfied by sets of assignments, called \emph{teams}, rather than by single assignments. It was originally developed by Hodges, in \cite{hodges97}, as a compositional alternative to the imperfect-information \emph{game theoretic semantics} for independence friendly logic \cite{hintikkasandu89,mann11}.

Over the past few years team semantics has been used to specify and study many other extensions of first-order logic. In particular, since a team describes a relation between the elements of its model team semantics offers a natural way to add to first-order logic atoms corresponding to database-theoretic \emph{dependency notions}.

This line of thought led first to the development of \emph{dependence logic} \cite{vaananen07}, and later to that of \emph{independence logic} \cite{gradel13} and \emph{inclusion and exclusion logics} \cite{galliani12}.\footnote{The literature contains many other extensions of first-order logic with team semantics, but we do not examine them in this work.}  By now there are many results in the literature concerning the properties of these logics, and in Section \ref{sect:prelims} we recall some of the principal ones.

One common characteristic of all these logics is that they are much stronger than first-order logic proper, even though they merely add \emph{first-order definable} dependency conditions to its language. 
Indeed, the rules of team semantics straddle the line between first and second order, since they evaluate first-order connectives by means of second-order machinery: and, while in the case of first-order logic formulas team semantics can be reduced to Tarski's semantics, if we add to our language atoms corresponding to further conditions the second-order nature of team semantics can take over.

The purpose of the present paper is to investigate the boundary between first and second order ``from below'', so to say, taking first-order logic with team semantics and trying to find out how much we can add to it while preserving first-orderness. In Section \ref{sect:upwards} we define a fairly general family of classes of first-order definable dependency conditions and prove they can be safely added to first-order logic; then in Section \ref{sect:upwards_const} we expand this family, and in Section \ref{sect:negdep} we show that, as a consequence, the negations of all the main dependency atoms studied in team semantics do not ``blow up'' first-order logic into a higher order one. Finally, in Section \ref{sect:bounded} we introduce a notion of \emph{boundedness} for dependencies and use it to demonstrate some \emph{non-definability results}.
\section{Preliminaries}
\label{sect:prelims}
In this section we will recall some fundamental definitions and results concerning team semantics.
\begin{defin}[Team]
Let $M$ be a first-order model and let $\dom(M)$ be the set of its elements.\footnote{We always assume that models have at least two elements in their domain.} Furthermore, let $V$ be a finite set of variables. Then a \emph{team} $X$ over $M$ with \emph{domain} $\dom(X) = V$ is a set of assignments $s$ from $V$ to $\dom(M)$.

Given a team $X$ and a tuple of variables $\tuple v$ contained in the domain of $X$, we write $X \upharpoonright \tuple v$ for the team obtained by restricting all assignments of $X$ to the variables of $\tuple v$ and $X(\tuple v)$ for the relation $\{s(\tuple v) : s \in X\} \subseteq \dom(M)^{|\tuple v|}$.
\end{defin}
As it is common when working with team semantics, we will assume that all our expressions are in negation normal form.
\begin{defin}[Team Semantics for First-Order Logic]
\label{def:TS}
Let $\phi(\tuple x)$ be a first-order formula in negation normal form with free variables in $\tuple x$. Furthermore, let $M$ be a first-order model whose signature contains the signature of $\phi$ and let $X$ be a team over it whose domain contains $\tuple x$. Then we say that $X$ \emph{satisfies} $\phi$ in $M$, and we write $M \models_X \phi$, if and only if this follows from these rules:\footnote{What we give here is the so-called \emph{lax} version of team semantics. There also exists a \emph{strict} version, with slightly different rules for disjunction and existential quantification; but as pointed out in \cite{galliani12}, \emph{locality} -- in the sense of Theorem \ref{thm:local} here -- fails in strict team semantics for some of the logics we are interested in. Therefore, in this work we will only deal with lax team semantics.}
\begin{description}
\item[TS-lit:] For all first-order literals $\alpha$, $M \models_X \alpha$ if and only if for all $s \in X$, $M \models_s \alpha$ according to the usual Tarski semantics; 
\item[TS-$\vee$:] For all $\psi$ and $\theta$, $M \models_X \psi \vee \theta$ if and only if $X = Y \cup Z$ for two subteams $Y$ and $Z$ such that $M \models_{Y} \psi$ and $M \models_Z \theta$; 
\item[TS-$\wedge$:] For all $\psi$ and $\theta$, $M \models_X \psi \wedge \theta$ if and only if $M \models_X \psi$ and $M \models_X \theta$; 
\item[TS-$\exists$:] For all $\psi$ and all variables $v$, $M \models_X \exists v \psi$ if and only if there exists a function 
\[
	H: X \rightarrow \parts(\dom(M)) \backslash \{\emptyset\}
\] 
such that $M \models_{X[H/v]} \psi$, where $X[H/v] = \{s[m/v] : s \in X, m \in H(s)\}$ and $\parts(\dom(M))$ is the powerset of $\dom(M)$;
\item[TS-$\forall$:] For all $\psi$ and all variables $v$, $M \models_X \forall v \psi$ if and only if $M \models_{X[M/v]} \psi$, where $X[M/v] = \{s[m/v] : s \in X, m \in M\}$.
\end{description}

Given a sentence (that is, a formula with no free variables) $\phi$ and a model $M$ over its signature, we say that $\phi$ is \emph{true} in $M$ and we write $M \models \phi$ if and only if $M \models_{{\{\emptyset\}}} \phi$.\footnote{Of course, one should not confuse the team $\{\emptyset\}$, which contains only the empty assignment, with the \emph{empty team} $\emptyset$, which contains no assignments at all.}
\end{defin}
The following is a useful and easily derived rule:
\begin{lemma}
Let $\tuple v = v_1 \ldots v_n$ be a tuple of $n$ variables and let $\exists \tuple v \psi$ be a shorthand for $\exists v_1 \ldots \exists v_n \psi$. Then $M \models_X \exists \tuple v \psi$ if and only if there exists a function $H: X \rightarrow \parts(\dom(M)^n) \backslash \{\emptyset\}$ such that $M \models_{X[H/\tuple v]} \psi$, where $X[H/\tuple v] = \{s[\tuple m / \tuple v] : s \in X, \tuple m \in H(s)\}$.
\end{lemma}
With respect to first-order formulas, team semantics can be reduced to Tarski's semantics. Indeed,
\begin{propo}[\cite{hodges97,vaananen07}]
\label{propo:flat}
Let $\phi(\tuple x)$ be a first-order formula in negation normal form with free variables in $\tuple x$. Furthermore, let $M$ be a first-order model whose signature contains that of $\phi$, and let $X$ be a team over $M$ whose domain contains $\tuple x$. Then $M \models_X \phi$ if and only if, for all $s \in X$, $M \models_s \phi$ with respect to Tarski's semantics. 

In particular, a first-order sentence $\phi$ is true in a model $M$ with respect to team semantics if and only if it is true in $M$ with respect to Tarski's semantics.
\end{propo}
Therefore, not all first-order definable properties of relations correspond to the satisfaction conditions of first-order formulas: for example, the non-emptiness of a relation $R$ is definable by $\exists \tuple x R \tuple x$, but there is no first order $\phi$ such that $M \models_X \phi$ if and only if $X \not = \emptyset$. More in general, let $\phi^*(R)$ be a first-order sentence specifying a property of the $k$-ary relation $R$ and let $\tuple x = x_1 \ldots x_k$ be a tuple of new variables: then, as it follows easily from the above proposition, there exists a first-order formula $\phi(\tuple x)$ such that 
\[
	M \models_X \phi(\tuple x) \Leftrightarrow M, X(\tuple x) \models \phi^*(R)
\]
if and only if $\phi^*(R)$ can be put in the form $\forall \tuple x (R \tuple x \rightarrow \theta(\tuple x))$ for some $\theta$ in which $R$ does not occur.\footnote{That is, according to the terminology of \cite{vaananen07}, if and only if $\phi^*(R)$ is \emph{flat}.}
%

It is hence possible to extend first-order logic (with team semantics) by introducing new atoms corresponding to further properties of relations. Database theory is a most natural choice as a source for such properties; and, in the rest of this section, we will recall the fundamental database-theoretic extensions of first-order logic with team semantics and some of their properties.

\textbf{Dependence logic} $\FO(=\!\!(\cdot, \cdot))$, from \cite{vaananen07}, adds to first-order logic \emph{functional dependence atoms} \\$=\!\!(\tuple x, \tuple y)$ based on database-theoretic \emph{functional dependencies} (\cite{armstrong74}). Their rule in team semantics is
\begin{description}
\item[TS-fdep:] $M \models_X =\!\!(\tuple x, \tuple y)$ if and only if for all $s, s' \in X$,  $s(\tuple x) = s'(\tuple x) \Rightarrow s(\tuple y) = s'(\tuple y)$.
\end{description}
This atom, and dependence logic as a whole, is \emph{downwards closed}: for all dependence logic formulas $\phi$, models $M$ and teams $X$, if $M \models_X \phi$ then $M \models_Y \phi$ for all $Y \subseteq X$. It is not however \emph{union closed}: if $M \models_X \phi$ and $M \models_Y \phi$ then we cannot in general conclude that $M \models_{X \cup Y} \phi$.

Dependence logic is equivalent to existential second-order logic over sentences:
\begin{theo}[\cite{vaananen07}]
Every dependence logic sentence $\phi$ is logically equivalent to some $\ESO$ sentence $\phi^*$, and vice versa.
\end{theo}
\textbf{Constancy logic} $\FO(=\!\!(\cdot))$ is the fragment of dependence logic which only allows functional dependence atoms of the form $=\!\!(\emptyset, \tuple x)$, which we will abbreviate as $=\!\!(\tuple x)$ and call \emph{constancy atoms}. Clearly we have that 
\begin{description}
\item[TS-const:] $M \models_X =\!\!(\tuple x)$ if and only if for all $s, s' \in X$, $s(\tuple x) = s'(\tuple x)$. 
\end{description}
As proved in \cite{galliani12}, every constancy logic sentence is equivalent to some first-order sentence: therefore, constancy logic is strictly weaker than dependence logic. Nonetheless, constancy logic is more expressive than first-order logic with respect to the second-order relations generated by the satisfaction conditions of formulas: indeed, it is an easy consequence of Proposition \ref{propo:flat} that no first-order formula is logically equivalent to the constancy atom $=\!\!(x)$. 

\textbf{Exclusion logic} $\FO(|)$, from \cite{galliani12}, adds to first-order logic \emph{exclusion atoms} $\tuple x ~|~ \tuple y$, where $\tuple x$ and $\tuple y$ are tuples of variables of the same length. Just as functional dependence atoms correspond to functional database-theoretic dependencies, exclusion atoms correspond to \emph{exclusion dependencies} \cite{casanova83}; and their satisfaction rule is
\begin{description}
\item[TS-excl:] $M \models_X \tuple x ~|~ \tuple y$ if and only if $X(\tuple x) \cap X(\tuple y) = \emptyset$.
\end{description}
As proved in \cite{galliani12}, exclusion logic is entirely equivalent to dependence logic: every exclusion logic formula is logically equivalent to some dependence logic formula, and vice versa. 

\textbf{Inclusion logic} $\FO(\subseteq)$, also from \cite{galliani12}, adds instead to first-order logic \emph{inclusion atoms} $\tuple x \subseteq \tuple y$ based on database-theoretic \emph{inclusion dependencies} \cite{fagin81}. The corresponding rule is 
\begin{description}
\item[TS-inc:] $M \models_X \tuple x \subseteq \tuple y$ if and only if $X(\tuple x) \subseteq X(\tuple y)$. 
\end{description}
Inclusion logic is stronger than first-order logic, but weaker than existential second-order logic: indeed, as shown in \cite{gallhella13}, sentence-wise it is equivalent to positive greatest fixed point logic GFP$^+$. Formula-wise, it is incomparable with constancy, dependence or exclusion logic, since its formulas are union closed but not downwards closed.

\textbf{Independence logic} $\FO(\bot)$, from \cite{gradel13}, adds to first-order logic \emph{independence atoms} $\indep{\tuple x}{\tuple y}$ with the intended meaning of ``the values of $\tuple x$ and $\tuple y$ are informationally independent''. More formally, 
\begin{description}
\item[TS-ind:] $M \models_X \indep{\tuple x}{\tuple y}$ if and only if $X(\tuple x \tuple y) = X(\tuple x) \times X(\tuple y)$. 
\end{description}
This notion of informational independence has a long history: see for example \cite{geiger91} for an analysis of this concept from a probabilistic perspective.

The \emph{conditional} independence atoms $\indepc{\tuple x}{\tuple y}{\tuple z}$, also from \cite{gradel13}, relativize the independence of $\tuple x$ and $\tuple y$ to all \emph{fixed} value of $\tuple z$. Their semantics is
\begin{description}
\item[TS-c-ind:] $M \models_X \indepc{\tuple x}{\tuple y}{\tuple z}$ if and only if for all tuples $\tuple m \in \dom(M)^{|\tuple z|}$ and for $X_{\tuple z = \tuple m} = \{s \in X : s(\tuple z) = \tuple m\}$ it holds that $X_{\tuple z = \tuple m}(\tuple x \tuple y) = X_{\tuple z = \tuple m}(\tuple x) \times X_{\tuple z = \tuple m}(\tuple y)$.
\end{description}
As pointed out in \cite{engstrom12}, the rule for $\indepc{\tuple x}{\tuple y}{\tuple z}$ corresponds precisely to the database-theoretic \emph{embedded multivalued dependency} \cite{fagin77} $(\tuple z \twoheadrightarrow \tuple x | \tuple y)$.

In \cite{gradel13} it was shown that every dependence logic formula is equivalent to some $FO(\bot_c)$ (conditional independence logic) formula, but not vice versa; and sentence-wise, both of these logics are equivalent to each other (and to $\ESO$). Furthermore, in \cite{galliani12} it was proved that $\FO(\bot_c)$ is equivalent to \textbf{inclusion/exclusion logic}\footnote{That is, to first-order logic plus inclusion \emph{and} exclusion atoms.} $\FO(\subseteq, |)$, even with respect to open formulas, and that this is, roughly speaking, the most general logic obtainable by adding first-order (or even existential second-order) definable dependency conditions to first-order logic.\footnote{To be more precise, for every $\ESO$ formula $\phi^*(R)$ there exists a $\FO(\bot_c)$ formula $\phi(\tuple x)$ such that, for all suitable models $M$ and nonempty teams $X$, $M \models_X \phi(\tuple x)$ if and only if $M, X(\tuple x) \models \phi^*(R)$.} More recently, in \cite{vaananen13}, it was shown that $\FO(\bot)$ and $\FO(\bot_c)$ are also equivalent.

We conclude this section with Figure \ref{fig1}, which depicts the relations between the logics we discussed so far.
\begin{figure}
\begin{center}
\epsfig{file=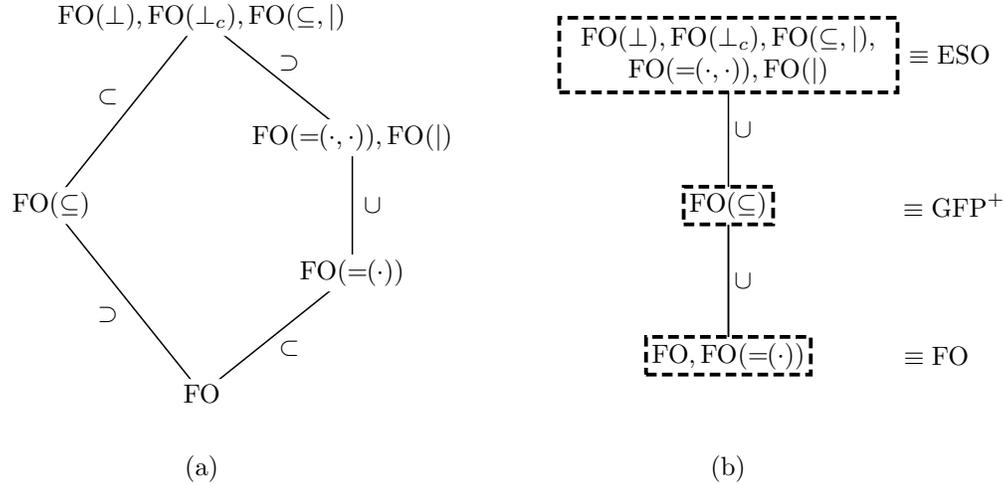}
\end{center}
\caption{Relations between logics wrt formulas (a) and sentences (b).}
\label{fig1}
\end{figure}
\section{Upwards Closed Dependencies}
\label{sect:upwards}
In this work we will study the properties of the logics obtained by adding families of \emph{dependency conditions} to the language of first-order logic. But what is a dependency condition, in a general sense? The following definition is based on the \emph{generalized atoms} of \cite{kuusisto13}:
\begin{defin}
Let $n \in \mathbb N$. A \emph{dependency} of arity $n$ is a class $\mathbf D$, closed under isomorphisms, of models over the signature $\{R\}$ where $R$ is a $n$-ary relation symbol. If $\tuple x$ is a tuple of $n$ variables (possibly with repetitions), $M$ is a first-order model and $X$ is a team over it whose domain contains all variables of $\tuple x$ then
\begin{description}
\item[TS-D:] $M \models_X \mathbf D \tuple x$ if and only if $(\dom(M), X(\tuple x)) \in \mathbf D$. 
\end{description}
\end{defin}
\begin{defin}
Let $\D = \{\mathbf D_1, \mathbf D_2, \ldots\}$ be a family of dependencies. Then we write $\FO(\D)$ for the logic obtained by adding to the language of first-order logic all \emph{dependency atoms} $\mathbf D \tuple x$, where $\mathbf D \in \D$ and $\tuple x$ is a tuple of variables of the arity of $\mathbf D$.
\end{defin}
It is not difficult to represent the logics of Section \ref{sect:prelims} in this notation. For example, dependence logic is $\FO(=\!\!(\cdot, \cdot))$ for $=\!\!(\cdot, \cdot) = \{=\!\!(n, m) : n, m \in \mathbb N\}$, where $(\dom(M), R) \in ~=\!\!(n,m)$ if and only if 
\[
\tuple a\tuple b, \tuple a\tuple c \in R \Rightarrow \tuple b = \tuple c
\]
for all tuples of elements $\tuple a = a_1 \ldots a_n$, $\tuple b = b_1 \ldots b_m$, $\tuple c = c_1 \ldots c_m \in \dom(M)$.

The following property can be easily verified, by induction on the formulas $\phi$:\footnote{For the sake of reference, we mention Theorem 4.22 of \cite{galliani12} in which the same result is proved in detail for (conditional) independence logic. The only new case here is the one in which $\phi(\tuple x) = \mathbf D \tuple y$ for some $\mathbf D \in \D$ and $\tuple y$ is contained in $\tuple x$; and for it the result follows at once from condition \textbf{TS-D} and from the fact that $X(\tuple y) = (X \upharpoonright \tuple x)(\tuple y)$.}
\begin{theo}[Locality]
\label{thm:local}
Let $\D$ be a family of dependencies and let $\phi(\tuple x)$ be a formula of $\FO(\D)$ with free variables in $\tuple x$. Then for all models $M$ and all teams $X$ over it whose domain contains $\tuple x$, $M \models_X \phi(\tuple x)$ if and only if $M \models_{X\upharpoonright \tuple x} \phi(\tuple x)$.
\end{theo}

In this work, we will be mainly interested in dependencies which correspond to first-order definable properties of relations:
\begin{defin}
A dependency notion $\mathbf D$ is \emph{first-order definable} if there exists a first-order sentence $\mathbf D^*(R)$ over the signature $\{R\}$, where $R$ is a new relation symbol, such that 
\[
	M \in \mathbf D \Leftrightarrow M \models \mathbf D^*(R)
\]
for all models $M = (\dom(M), R)$.
\end{defin}
It is not necessarily the case that if $\mathbf D$ is first-order definable then $\FO(\mathbf D)$ and $\FO$ are equivalent with respect to sentences. For example $=\!\!(n, m)^*(R)$ is $\forall \tuple x \tuple y \tuple z (R\tuple x\tuple y \wedge R\tuple x\tuple z \rightarrow \tuple y = \tuple z)$, where $\tuple x$ has length $n$ and $\tuple y, \tuple z$ have length $m$; but as we said in Section \ref{sect:prelims}, dependence logic is stronger than first-order logic.

When is then the case that dependency conditions can be added safely to first-order logic, without increasing the expressive power? The following definition will provide us a partial answer:
\begin{defin}
A dependency notion $\mathbf D$ is \emph{upwards closed} if
\[
	(\dom(M), R) \in \mathbf D, R \subseteq S \Rightarrow (\dom(M), S) \in \mathbf D
\]
for all models $(\dom(M), R)$ and all relations $S$ over $\dom(M)$ of the same arity of $R$.
\end{defin}
It is easy to see that upwards closed dependencies induce upwards closed satisfaction rules: if $\mathbf D$ is upwards closed, $M \models_X \mathbf D \tuple x$ and $X \subseteq Y$ then it is always the case that $M \models_Y \mathbf D \tuple x$. However, differently from the case of downwards or union closure, upwards closure is not preserved by team semantics: if $\mathbf D$ is upwards closed, $\phi \in \FO(\mathbf D)$ and $M \models_X \phi$ then it is not in general true that $M \models_Y \phi$ for all $Y \supseteq X$ (for example, let $\phi$ be a nontrivial first-order literal and recall Rule \textbf{TS-lit}).

Some examples of upwards closed dependencies follow:
\begin{description}
\item[Non-emptiness:] $M \models_X \nonempty$ if and only if $X \not = \emptyset$; 
\item[Intersection:] $M \models_X \Diamond(\tuple x = \tuple y)$ if and only if there exists a $s \in X$ with $s(\tuple x) = s(\tuple y)$;
\item[Inconstancy:] $M \models_X \not =\!\!(\tuple x)$ if and only if $|X(\tuple x)| > 1$; 
\item[$n$-bigness:] For all $n \in \mathbb N$, $M \models_X |\tuple x| \geq n$ if and only if $|X(\tuple x)| \geq n$;
\item[Totality:] $M \models_X \all(\tuple x)$ if and only if $X(\tuple x) = \dom(M)^{|\tuple x|}$; 
\item[Non-dependence:] $M \models_X \not= \!\!(\tuple x, \tuple y)$ if and only if there exist $s, s' \in X$ with $s(\tuple x) = s'(\tuple x)$ but $s(\tuple y) \not = s'(\tuple y)$;\footnote{The same symbol $\not =\!\!(\tuple x, \tuple y)$ has been used in \cite{galliani12c} to describe a different non-dependence notion, stating that for \emph{every} $s \in X$ there exists a $s' \in X$ with $s(\tuple x) = s'(\tuple x), s(\tuple y) \not= s'(\tuple y)$. In that thesis it was proved that the resulting ``non-dependence logic'' is equivalent to inclusion logic. As we will see, this is not the case for the non-dependence notion of this paper.}
\item[Non-exclusion:] $M \models_X \tuple x \nmid \tuple y$ if and only if there exist $s, s' \in X$ with $s(\tuple x) = s'(\tuple y)$; 
\item[Infinity:] $M \models_X |\tuple x| \geq \omega$ if and only if $X(\tuple x)$ is infinite; 
\item[$\kappa$-bigness:] For all cardinals $\kappa$, $M \models_X |\tuple x| \geq \kappa$ if and only if $|X(\tuple x)| \geq \kappa$.
\end{description}
All the above examples except infinity and $\kappa$-bigness are first-order definable. The $\nonempty$ atom is the adaptation to first-order team semantics of the non-emptiness atom introduced in \cite{yang13b} for the propositional version of dependence logic, and the totality atom $\all$ is due to Abramsky and V\"a\"an\"anen (\cite{abramsky13}).

The main result of this section is the following:
\begin{theo}
\label{thm:upwards}
Let $\D$ be a collection of upwards closed first-order definable dependency conditions. Then for every formula $\phi(\tuple x)$ of $FO(\D)$ with free variables in $\tuple x$ there exists a first-order sentence $\phi^*(R)$, where $R$ is a new $|\tuple x|$-ary relation symbol, such that 
\[
	M \models_X \phi(\tuple x) \Leftrightarrow M, X(\tuple x) \models \phi^*(R)
\]
for all models $M$ over the signature of $\phi$ and all teams $X$. 

In particular, every sentence of $\FO(\D)$ is equivalent to some first-order sentence.
\end{theo}

Let us begin by adapting the notion of \emph{flattening} of \cite{vaananen07} to the case of an arbitrary logic $\FO(\D)$:
\begin{defin}
Let $\D$ be any set of dependency conditions and let $\phi$ be a $FO(\D)$ formula. Then its \emph{flattening} $\phi^f$ is the first-order formula obtained by replacing any non-first-order atom with $\top$, where $\top$ is the trivially true atom.
\end{defin}
It is trivial to see, by induction on $\phi$, that 
\begin{lemma}
\label{lem:flat_imp}
For all $\D$, all $\phi \in \FO(\D)$, all models $M$ and all teams $X$ over $M$, if $M \models_X \phi$ then $M \models_X \phi^f$.
\end{lemma}
As we said, even if $\D$ contains only upwards closed dependency conditions it is not true that all formulas of $\FO(\D)$ are upwards closed. However, the following restricted variant of upwards closure is preserved:
\begin{theo}
\label{thm:upflat}
Let $\phi$ be a $FO(\D)$ formula, where $\D$ contains only upwards closed dependencies. Let $M$ be a first-order model, and let $X$, $Y$ be teams such that $X \subseteq Y$, $M \models_X \phi$, and $M \models_Y \phi^f$. Then $M \models_Y \phi$. 
\end{theo}
\begin{proof}
The proof is by structural induction on $\phi$. 
\begin{enumerate}
\item If $\phi$ is a first-order literal, $\phi^f = \phi$ and there is nothing to prove; 
\item If $\phi$ is of the form $\mathbf D \tuple x$ for some $\mathbf D \in \D$, $M \models_X \phi$ and $X \subseteq Y$, then by upwards closure $M \models_Y \phi$;
\item Suppose that $M \models_X \phi_1 \vee \phi_2$ and $M \models_Y \phi_1^f \vee \phi_2^f$. Now $X = X_1 \cup X_2$ for two $X_1$, $X_2$ such that $M \models_{X_1} \phi_1$ and $M \models_{X_2} \phi_2$, and therefore by Lemma \ref{lem:flat_imp} $M \models_{X_1} \phi_1^f$ and $M \models_{X_2} \phi_2^f$. Furthermore, $Y = Y_1 \cup Y_2$ for two $Y_1$, $Y_2$ such that $M \models_{Y_1} \phi^f$ and $M \models_{Y_2} \phi_2^f$. Let $Z_1 = X_1 \cup Y_1$ and $Z_2 = X_2 \cup Y_2$; then $Z_1 \cup Z_2 = X \cup Y = Y$, and by Proposition \ref{propo:flat} $M \models_{Z_1} \phi_1^f$ and $M \models_{Z_2} \phi_2^f$. But $M \models_{X_1} \phi_1$ and $X_1 \subseteq Z_1$, so by induction hypothesis $M \models_{Z_1} \phi_1$; and similarly, $M \models_{X_2} \phi_2$ and $X_2 \subseteq Z_2$, so $M \models_{Z_2} \phi_2$. Therefore $M \models_Y \phi_1 \vee \phi_2$, as required. 
\item If $M \models_{X} \phi_1 \wedge \phi_2$ then $M \models_X \phi_1$ and $M \models_X \phi_2$. Then by induction hypothesis, since $M \models_Y \phi_1^f$ and $X \subseteq Y$, $M \models_Y \phi_1$; and similarly, since $M \models_Y \phi_2^f$ and $X \subseteq Y$, $M \models_Y \phi_2$, and therefore $M \models_Y \phi_1 \wedge \phi_2$.
\item If $M \models_X \exists v \phi$ then there is a function $H: X \rightarrow \parts(\dom(M)) \backslash \{\emptyset\}$ such that $M \models_{X[H/v]} \phi$, and therefore (by Lemma \ref{lem:flat_imp}) such that $M \models_{X[H/v]} \phi^f$. Similarly, if $M \models_Y \exists v \phi^f$ then for some $K$ we have that $M \models_{Y[K/v]} \phi^f$. Now let $W: Y \rightarrow \parts(\dom(M)) \backslash \{\emptyset\}$ be such that
\[
	W(s) = \left\{\begin{array}{l l}
		H(s) \cup K(s) & \mbox{ if } s \in X;\\
		K(s) & \mbox{ if } s \in Y \backslash X.
	\end{array}\right.
\]
Then $Y[W/v] = X[H/v] \cup Y[K/v]$, and therefore by Proposition \ref{propo:flat} $M \models_{Y[W/v]} \phi^f$. Then by induction hypothesis $M \models_{Y[W/v]} \phi$, since $X[H/v]$ satisfies $\phi$ and is contained in $Y[W/v]$; and therefore $M \models_Y \exists v \phi$, as required. 
\item If $M \models_X \forall v \phi$ then $M \models_{X[M/v]} \phi$, and if $M \models_Y \forall v \phi^f$ then $M \models_{Y[M/v]} \phi^f$. Now $X[M/v] \subseteq Y[M/v]$, so by induction hypothesis $M \models_{Y[M/v]} \phi$, and therefore $M \models_Y \forall v \phi$.
\end{enumerate}
\end{proof}
\begin{defin}
If $\theta$ is a first-order formula and $\phi$ is a $FO(\D)$ formula we define $(\phi \upharpoonright \theta)$ as $(\lnot \theta) \vee (\theta \wedge \phi)$, where $\lnot \theta$ is a shorthand for the first-order formula in negation normal form which is equivalent to the negation of $\theta$.
\end{defin}
The following lemma is obvious:
\begin{lemma}
\label{lem:FOrest}
For all first order $\theta$ and $\phi \in FO(\D)$, $M \models_X (\phi \upharpoonright \theta)$ if and only if $M \models_Y \phi$ for $Y = \{ s \in X : M \models_s \theta\}$. 
\end{lemma}
One can observe that $(\phi \upharpoonright \theta)$ is logically equivalent to $\theta \hookrightarrow \phi$, where $\hookrightarrow$ is the \emph{maximal implication} of \cite{kontinennu09}:
\begin{description}
\item[TS-maximp:] $M \models_X \theta \hookrightarrow \phi$ if and only if for all maximal $Y \subseteq X$ s.t. $M \models_Y \theta$, $M \models_Y \phi$.
\end{description}
We use the notation $(\phi \upharpoonright \theta)$, instead of $\theta \hookrightarrow \phi$, to make it explicit that $\theta$ is first order and that Lemma \ref{lem:FOrest} holds.

The next step of our proof of Theorem \ref{thm:upwards} is to identify a fragment of our language whose satisfaction conditions do not involve quantification over second-order objects such as teams or functions. We do so by limiting the availability of disjunction and existential quantification:
\begin{defin}
A $FO(\D)$ formula $\phi$ is \emph{clean} if
\begin{enumerate}
\item All its disjunctive subformulas $\psi_1 \vee \psi_2$ are first order or of the form $\psi \upharpoonright \theta$ for some suitable choice of $\psi$ and $\theta$ (where $\theta$ is first order);
\item All its existential subformulas $\exists v \psi$ are first order.
\end{enumerate}
\end{defin}
As the next proposition shows, clean formulas correspond to first-order definable properties of relations. 
\begin{propo}
\label{propo:detNF}
Let $\D$ be a class of first-order definable dependencies and let $\phi(\tuple x) \in \FO(\D)$ be a clean formula with free variables in $\tuple x$. Then there exists some first-order sentence $\phi^*(R)$, where $R$ is a new $|\tuple x|$-ary relation, such that
\begin{equation}
\label{eq:repFO}
	M \models_X \phi(\tuple x) \Leftrightarrow M, X(\tuple x) \models \phi^*(R).
\end{equation}
\end{propo}
\begin{proof}
By induction over $\phi$. 
\begin{enumerate}
\item If $\phi(\tuple x)$ is a first-order formula (not necessarily just a literal) then let $\phi^*(R) = \forall \tuple x (R \tuple x \rightarrow \phi(\tuple x))$. By Proposition \ref{propo:flat}, (\ref{eq:repFO}) holds. 
\item If $\phi(\tuple x)$ is a dependency atom $\mathbf D \tuple y$, where $\mathbf D \in \D$ and $\tuple y$ is a tuple (possibly with repetitions) of variables occurring in $\tuple x$, let $\phi^*(R)$ be obtained from $\mathbf D^*(S)$ by replacing every instance $S \tuple z$ of $S$ in it with $\exists \tuple x (\tuple z = \tuple y \wedge R \tuple x)$. Indeed, $M \models_X \mathbf D \tuple y$ if and only if $M, X(\tuple y) \models \mathbf D^*(S)$, and $\tuple m \in X(\tuple y)$ if and only if $M, X(\tuple x) \models \exists \tuple x  (\tuple m = \tuple y \wedge R \tuple x)$.
\item If $\phi(\tuple x)$ is of the form $(\psi(\tuple x) \upharpoonright \theta(\tuple x))$, let $\phi^*(R)$ be obtained from $\psi^*(R)$ by replacing every instance $R \tuple z$ of $R$ with $R \tuple z \wedge \theta(\tuple z)$. Indeed, by Lemma \ref{lem:FOrest} $M \models_X (\psi(\tuple x) \upharpoonright \theta(\tuple x))$ if and only if $M \models_Y \psi(\tuple x)$ for $Y = \{s \in X: M \models_s \theta\}$, and $\tuple m \in Y(\tuple x) \Leftrightarrow \tuple m \in X(\tuple x) \mbox{ and } M \models \theta(\tuple m)$. 
\item If $\phi(\tuple x)$ is of the form $\psi(\tuple x) \wedge \theta(\tuple x)$ simply let $\phi^*(R) = \psi^*(R) \wedge \theta^*(R)$.
\item If $\phi(\tuple x)$ is of the form $\forall v \psi(\tuple x, v)$, where we assume without loss of generality that $v$ is distinct from all $x \in \tuple x$, and $\psi^*(S)$ corresponds to $\psi(\tuple x, v)$ then let $\phi^*(R)$ be obtained from $\psi^*(S)$ by replacing every $S \tuple z w$ with $R \tuple z$. Indeed, $M \models_X \forall v \psi$ if and only if $M \models_{X[M/v]} \psi(\tuple x , v)$ and $\tuple m m' \in X[M/v](\tuple x v)$ if and only if $\tuple m \in X(\tuple x)$. 
\end{enumerate}
\end{proof}
All that is now left to prove is the following:
\begin{propo}
\label{propo:upwards}
Let $\D$ be a family of upwards closed dependencies. Then every $\FO(\D)$ formula is equivalent to some clean $\FO(\D)$ formula. 
\end{propo}
\begin{proof}
It suffices to observe the following facts: 
\begin{itemize}
\item If $\phi_1(\tuple x)$ and $\phi_2(\tuple x)$ are in $\FO(\D)$ then $\phi_1(\tuple x) \vee \phi_2(\tuple x)$ is logically equivalent to 
\[
	(\phi_1^f \vee \phi_2^f) \wedge (\phi_1 \upharpoonright \phi_1^f) \wedge (\phi_2 \upharpoonright \phi_2^f).
\]
Indeed, suppose that $M \models_X \phi_1 \vee \phi_2$: then, by Lemma \ref{lem:flat_imp}, $M \models_X \phi_1^f \vee \phi_2^f$. Furthermore, $X = Y \cup Z$ for two $Y$ and $Z$ such that $M \models_Y \phi_1$ and $M \models_Z \phi_2$. Now let $Y' = \{s \in X : M \models_s \phi_1^f\}$ and $Z' = \{s \in X : M \models_s \phi_2^f\}$:  by Lemma \ref{lem:flat_imp} and Proposition \ref{propo:flat} we have that $Y \subseteq Y'$ and that $Z \subseteq Z'$, and therefore by Theorem \ref{thm:upflat} $M \models_{Y'} \phi_1$ and $M \models_{Z'} \phi_2$. Thus by Lemma \ref{lem:FOrest} $M \models_X (\phi_1 \upharpoonright \phi_1^f)$ and $M \models_X (\phi_2\upharpoonright \phi_2^f)$, as required. 

Conversely, suppose that $M \models_X (\phi_1^f \vee \phi_2^f) \wedge (\phi_1 \upharpoonright \phi_1^f) \wedge (\phi_2 \upharpoonright \phi_2^f)$. Then let $Y = \{s \in X : M \models_s \phi_1^f\}$ and $Z = \{s \in X : M \models_s \phi_2^f\}$. By Proposition \ref{propo:flat} and since $M \models_X \phi_1^f \vee \phi_2^f$, $X = Y \cup Z$; and by Lemma \ref{lem:FOrest}, $M \models_Y \phi_1$ and $M \models_Z \phi_2$. So $M \models_X \phi_1 \vee \phi_2$, as required.
\item If $\phi(\tuple x, v) \in \FO(\D)$ then $\exists v \phi(\tuple x, v)$ is logically equivalent to 
\[
	(\exists v \phi^f(\tuple x, v)) \wedge \forall v (\phi(\tuple x, v) \upharpoonright \phi^f(\tuple x, v)).
\]
Indeed, suppose that $M \models_X \exists v \phi(\tuple x, v)$. Then by Lemma \ref{lem:flat_imp} $M \models_X \exists v \phi^f(\tuple x, v)$. Furthermore, for some $H: X \rightarrow \parts(\dom(M)) \backslash \{\emptyset\}$ and for $Y = X[H/v]$ it holds that $M \models_{Y} \phi(\tuple x, v)$. Now let $Z = \{h \in X[M/v] : M \models_h \phi^f(\tuple x, v)\}$. By Proposition \ref{propo:flat}, $M \models_Z \phi^f(\tuple x, v)$; and since $Y \subseteq Z$, by Theorem \ref{thm:upflat} $M \models_Z \phi(\tuple x, v)$, and therefore by Lemma \ref{lem:FOrest} $M \models_{X[M/v]} (\phi(\tuple x, v) \upharpoonright \phi^f(\tuple x, v))$, as required. 

Conversely, suppose that $M \models_X (\exists v \phi^f(\tuple x, v)) \wedge \forall v (\phi(\tuple x, v) \upharpoonright \phi^f (\tuple x, v))$. Then, for all $s \in X$, let $K(s) = \{m \in \dom(M) : M \models_{s[m/v]} \phi^f(\tuple x, v)\}$. Since $M \models_X \exists v \phi^f(\tuple x, v)$, $K(s)$ is nonempty for all $s \in X$, and by construction $X[K/v] = \{s \in X[M/v] : M \models_s \phi^f(\tuple x, v)\}$. Now $M \models_{X[M/v]}(\phi(\tuple x, v) \upharpoonright \phi^f(\tuple x, v))$, so by Lemma \ref{lem:FOrest} $M \models_{X[K/v]} \phi(\tuple x, v)$ and in conclusion $M \models_X \exists v \phi(\tuple x, v)$.
\end{itemize}
Applying inductively these two results to all subformulas of some $\phi \in \FO(\D)$ we can obtain some clean $\phi'$ to which $\phi$ is equivalent, and this concludes the proof.
\end{proof}
Finally, the proof of Theorem \ref{thm:upwards} follows at once from Propositions \ref{propo:detNF} and \ref{propo:upwards}.

Since, as we saw, the negations of functional and exclusion dependencies are upwards closed, we obtain at once the following corollary: 
\begin{coro}
\label{coro:nexfu}
Any sentence of $\FO(\not =\!\!(\cdot, \cdot), \nmid)$ (that is, of first-order logic plus negated functional and exclusion dependencies) is equivalent to some first-order sentence.
\end{coro}
\section{Adding Constancy Atoms}
\label{sect:upwards_const}
As we saw in the previous section, upwards closed dependencies can be added to first-order logic without increasing its expressive power (with respect to sentences); and as mentioned in Section \ref{sect:prelims}, this is also true for the (non upwards-closed) constancy dependencies $=\!\!(\tuple x)$. 

But what if our logic contains both upwards closed \emph{and} constancy dependencies? As we will now see, the conclusion of Theorem \ref{thm:upwards} remains valid: 
\begin{theo}
\label{thm:upwards_const}
Let $\D$ be a collection of upwards closed first-order definable dependency conditions. Then for every formula $\phi(\tuple x)$ of\footnote{Here $=\!\!(\cdot)$ represents the class of all constancy dependencies of all arities. But it is easy to see that the one of arity 1 would suffice: indeed, if $\tuple x$ is $x_1 \ldots x_n$ then $=\!\!(\tuple x)$ is logically equivalent to $=\!\!(x_1) \wedge \ldots \wedge =\!\!(x_n)$.} $FO(=\!\!(\cdot), \D)$ with free variables in $\tuple x$ there exists a first-order sentence $\phi^*(R)$, where $R$ is a new $|\tuple x|$-ary relation symbol, such that 
\[
	M \models_X \phi(\tuple x) \Leftrightarrow M, X(\tuple x) \models \phi^*(R).
\]
In particular, every sentence of $\FO(\D)$ is equivalent to some first-order sentence.
\end{theo}
The main ingredient of our proof will be the following lemma.
\begin{lemma}
\label{lem:const}
Let $\D$ be any family of dependencies and let $\phi(\tuple x)$ be a $\FO(=\!\!(\cdot), \D)$ formula. Then $\phi(\tuple x)$ is equivalent to some formula of the form $\exists \tuple v (=\!\!(\tuple v) \wedge \psi(\tuple x, \tuple v))$, where $\psi \in \FO(\D)$ contains exactly the same instances of $\mathbf D$-atoms (for all $\mathbf D \in \D$) that $\phi$ does, and in the same number.
\end{lemma}
The proof of this lemma is by induction on $\phi$, and it is entirely analogous to the corresponding proof from \cite{galliani12}.

Now we can prove Theorem \ref{thm:upwards_const}. \\
\begin{proof}
Let $\phi(\tuple x)$ be a $\FO(=\!\!(\cdot), \D)$-formula. Then by Lemma \ref{lem:const} $\phi(\tuple x)$ is equivalent to some sentence of the form $\exists \tuple v (=\!\!(\tuple v) \wedge \psi(\tuple x, \tuple v))$, where $\psi(\tuple x, \tuple v) \in \FO(\D)$. But then by Theorem \ref{thm:upwards} there exists a first-order formula $\psi^*(S)$ such that $M \models_X \psi(\tuple x, \tuple v)$ if and only if $M, X(\tuple x \tuple v) \models \psi^*(S)$. Now let $\theta(R, \tuple v)$ be obtained from $\psi^*(S)$ by replacing any $S \tuple y \tuple z$ with $R \tuple y \wedge \tuple z = \tuple v$. Since $X[\tuple m /\tuple v](\tuple x \tuple v) = \{\tuple a \tuple m : \tuple a \in X(\tuple x)\}$ it is easy to see that $M \models_X \exists \tuple v (=\!\!(\tuple v) \wedge \psi(\tuple x, \tuple v))$ if and only if $M, X(\tuple x) \models \exists v \theta(R, \tuple v)$, and this concludes the proof.
\end{proof}
\section{Possibility, Negated Inclusion and Negated Conditional Independence} 
\label{sect:negdep}
By Corollary \ref{coro:nexfu}, the negations of exclusion and functional dependence atoms can be added to first-order logic without increasing its power. But what about the negations of inclusion and (conditional) independence? These are of course first-order definable, but they are not upwards closed: indeed, their semantic rules can be given as
\begin{description}
\item[TS-$\not \subseteq$:] $M \models_X \tuple x \not \subseteq \tuple y$ if and only if there is a $s \in X$ such that for all $s' \in X$, $s(\tuple x) \not = s'(\tuple y)$; 
\item[TS-$\not \!\!\bot_c$:] $M \models_X \nindepc{\tuple x}{\tuple y}{\tuple z}$ if and only if there are $s, s' \in X$ with $s(\tuple z) = s'(\tuple z)$ and such that for all $s'' \in X$, $s''(\tuple x \tuple z) \not = s(\tuple x \tuple z)$ or $s''(\tuple y \tuple z) \not = s(\tuple y \tuple z)$.
\end{description}
However, we will now prove that, nonetheless, $\FO(\not =\!\!(\cdot, \cdot), \not \subseteq, \nmid, \not \!\!\bot_c)$ is equivalent to $\FO$ on the level of sentences. In order to do so, let us first define the following \emph{possibility operator} and prove that it is uniformly definable in $\FO(=\!\!(\cdot), \not = \!\!(\cdot))$: 
\begin{defin}
Let $\phi$ be any $\FO(\D)$ formula, for any choice of $\D$. Then 
\begin{description}
\item[TS-$\Diamond$:] $M \models_X \Diamond \phi$ if there exists a $Y \subseteq X$, $Y \not = \emptyset$, such that $M \models_Y \phi$.
\end{description}
\end{defin}
\begin{lemma}
Let $\phi$ be any $\FO(\D)$ formula, for any $\D$. Then $\Diamond\phi$ is logically equivalent to 
\begin{equation}
\label{eq2}
\exists u_0 u_1 \exists v (=\!\!(u_0) \wedge =\!\!(u_1) \wedge (v = u_0 \vee v = u_1) \wedge (\phi \upharpoonright v = u_1) \wedge \not =\!\!(v)).
\end{equation}
\end{lemma}
\begin{proof}
Suppose that there is a $Y \subseteq X$, $Y \not = \emptyset$, such that $M \models_Y \phi$. Then let $0, 1 \in \dom(M)$ be such that $0 \not = 1$, let $H: X[01/u_0u_1] \rightarrow \parts(\dom(M)) \backslash \{\emptyset\}$ be such that 
\[
	H(s[01/u_0u_1]) = \left\{\begin{array}{l l}
		\{0, 1\} & \mbox{ if } s \in Y;\\
		\{0\} & \mbox{ if } s \in X \backslash Y
	\end{array}\right.
\]
and let $Z = X[01/u_0u_1][H/v]$. Clearly $M \models_Z =\!\!(u_0) \wedge =\!\!(u_1) \wedge (v = u_0 \vee v = u_1) \wedge (\phi \upharpoonright v = u_1)$, and it remains to show that $M \models_Z \not =\!\!(v)$. But by hypothesis $Y$ is nonempty, and therefore there exists a $s \in Y \subseteq X$ such that $\{s[010/u_0 u_1 v], s[011/u_0u_1v]\} \subseteq Z$. So $v$ is not constant in $Z$, as required, and $X$ satisfies (\ref{eq2}).

Conversely, suppose that $X$ satisfies (\ref{eq2}), let $0$ and $1$ be our choices for $u_0$ and $u_1$, and let $H$ be the choice function for $v$. Then let $Y = \{s \in X : 1 \in H(s[01/u_0u_1])\}$. By locality, Lemma \ref{lem:FOrest} and the fact that $M \models_{X[01H/u_1 u_2 v]} (\phi \upharpoonright v = u_1)$ we have that $M \models_Y \phi$; and $Y$ is nonempty, since\\$M \models_Z (v = u_0 \vee v = u_1) \wedge \not = \!\!(v)$. 
\end{proof}
It is now easy to see that the negations of inclusion and conditional independence are in $\FO(=\!\!(\cdot), \not =\!\!(\cdot))$:
\begin{propo}
For all $\tuple x$, $\tuple y$ with $|\tuple x| = |\tuple y|$, $\tuple x \not \subseteq \tuple y$ is logically equivalent to 
\[
	\exists \tuple z (=\!\!(\tuple z) \wedge \Diamond(\tuple z = \tuple x) \wedge \tuple z \not = \tuple y).
\]
\end{propo}
\begin{propo}
For all $\tuple x$, $\tuple y$ and $\tuple z$, $\nindepc{\tuple x}{\tuple y}{\tuple z}$ is logically equivalent to 
\[
	\exists \tuple p \tuple q \tuple r (=\!\!(\tuple p \tuple q \tuple r) \wedge \Diamond(\tuple p \tuple r = \tuple x \tuple z) \wedge \Diamond(\tuple q \tuple r = \tuple y \tuple z) \wedge \tuple p \tuple q \tuple r \not = \tuple x \tuple y \tuple z).
\]
\end{propo}
\begin{coro}
Every sentence of $\FO(\not =\!\!(\cdot, \cdot), \not \subseteq, \nmid, \not \!\! \bot_c)$ is equivalent to some sentence of \\$\FO(=\!\!(\cdot), \not =\!\!(\cdot, \cdot), \nmid)$, and hence to some first-order sentence.
\end{coro}
\section{Bounded Dependencies and Totality}
\label{sect:bounded}
Now that we know something about upwards closed dependencies, it would be useful to classify them in different categories and prove \emph{non-definability} results between the corresponding extensions of first-order logic. As a first such classification, we introduce the following property:
\begin{defin}[Boundedness]
Let $\kappa$ be a (finite or infinite) cardinal. A dependency condition $\mathbf D$ is $\kappa$-\emph{bounded} if whenever $M \models_X \mathbf D \tuple x$ there exists a $Y \subseteq X$ with $|Y| \leq \kappa$ such that $M \models_Y \mathbf D \tuple x$.

We say that $\mathbf D$ is \emph{bounded} if it is $\kappa$-bounded for some $\kappa$.\footnote{After a fashion, this notion of boundedness may be thought of as a dual of the notion of \emph{coherence} of \cite{kontinen_ja13}.}
\end{defin}
For example, non-emptiness and intersection are $1$-bounded; inconstancy and the negations of functional dependence and exclusion are $2$-bounded; and for all finite or infinite $\kappa$, $\kappa$-bigness is $\kappa$-bounded. However, totality is not bounded at all. Indeed, for any $\kappa$ consider a model $M$ of cardinality greater than $\kappa$ and take the team $X=\{\emptyset\}[M/x]$. Then $M \models_X \all(x)$, but if $Y \subseteq X$ has cardinality $\leq \kappa$ then $Y(x) \subsetneq \dom(M)$ and $M \not \models_Y \all(x)$.

As we will now see, the property of boundedness is preserved by the connectives of our language. 
\begin{defin}[Height of a formula]
Let $\D$ be any family of bounded dependencies. Then for all formulas $\phi \in \FO(\D)$, the \emph{height} $\height(\phi)$ of $\phi$ is defined as follows:
\begin{enumerate}
\item If $\phi$ is a first-order literal then $\height(\phi) = 0$; 
\item If $\phi$ is a functional dependence atom $\mathbf D \tuple x$ then $\height(\phi)$ is the least cardinal $\kappa$ such that $\mathbf D$ is $\kappa$-bounded; 
\item If $\phi$ is of the form $\psi_1 \vee \psi_2$ or $\psi_1 \wedge \psi_2$ then $\height(\phi) = \height(\psi_1) + \height(\psi_2)$;
\item If $\phi$ is of the form $\exists v \psi$ or $\forall v \psi$. then $\height(\phi) = \height(\psi)$. 
\end{enumerate}
\end{defin}
In other words, the height of a formula is the sum of the heights of all instances of dependency atoms occurring in it. 
\begin{theo}
\label{thm:height}
Let $\D$ be a family of bounded upwards closed dependencies. Then for all formulas $\phi \in \FO(\D)$
\[
	M \models_X \phi \Rightarrow \exists Y \subseteq X \mbox{ with } |Y| \leq \height(\phi) \mbox{ s.t. } M \models_Y \phi.
\]
\end{theo}
\begin{proof}
The proof is by induction on $\phi$. 
\begin{enumerate}
\item If $\phi$ is a first-order literal then $\height(\phi) = 0$ and it is always the case that $M \models_\emptyset \phi$, as required.
\item If $\phi$ is an atom $\mathbf D \tuple x$ then the statement follows at once from the definitions of boundedness and height.
\item If $\phi$ is a disjunction $\psi_1 \vee \psi_2$ then $\height(\phi) = \height(\psi_1) + \height(\psi_2)$. Suppose now that $M \models_X \psi_1 \vee \psi_2$: then $X = X_1 \cup X_2$ for two $X_1$ and $X_2$ such that $M \models_{X_1} \psi_1$ and $M \models_{X_2} \psi_2$. This implies that there exist $Y_1 \subseteq X_1$, $Y_2 \subseteq X_2$ such that $M \models_{Y_1} \psi_1$ and $M \models_{Y_2} \psi_2$, $|Y_1| \leq \height(\psi_1)$ and $|Y_2| \leq \height(\psi_2)$. But then $Y = Y_1 \cup Y_2$ satisfies $\psi_1 \vee \psi_2$ and has at most $\height(\psi_1) + \height(\psi_2)$ elements. 
\item If $\phi$ is a conjunction $\psi_1 \wedge \psi_2$ then, again, $\height(\phi) = \height(\psi_1) + \height(\psi_2)$. Suppose that $M \models_X \psi_1 \wedge \psi_2$: then $M \models_X \psi_1$ and $M \models_X \psi_2$, and therefore by Lemma \ref{lem:flat_imp} $M \models_X \psi_1^f$ and $M \models_X \psi_2^f$; and, by induction hypothesis, there exist $Y_1, Y_2 \subseteq X$ with $|Y_1| \leq \height(\psi_1)$, $|Y_2| \leq \height(\psi_2)$, $M \models_{Y_1} \psi_1$ and $M \models_{Y_2} \psi_2$. Now let $Y = Y_1 \cup Y_2$: since $Y \subseteq X$, by Proposition \ref{propo:flat} $M \models_Y \psi_1^f$ and $M \models_Y \psi_2^f$. But $Y_1, Y_2 \subseteq Y$, and therefore by Theorem \ref{thm:upflat} $M \models_Y \psi_1$ and $M \models_Y \psi_2$, and in conclusion $M \models_Y \psi_1 \wedge \psi_2$.
\item If $\phi$ is of the form $\exists v \psi$ then $\height(\phi) = \height(\psi)$. Suppose that $M \models_X \exists v \psi$: then for some $H$ we have that $M \models_{X[H/v]} \psi$, and therefore by induction hypothesis there exists a $Z \subseteq X[H/v]$ with $|Z| \leq \height(\psi)$ such that $M \models_{Z} \psi$. For any $h \in Z$, let $\mathfrak f (h)$ be a $s \in X$ such that $h \in s[H/v] = \{s[m/v] : m \in H(s)\}$,\footnote{Since $Z \subseteq X[H/v]$, such a $s$ always exists. Of course, there may be multiple ones; in that case, we pick one arbitrarily.} and let $Y = \{\mathfrak f(h) : h \in Z\}$. Now $Z \subseteq Y[H/v] \subseteq X[H/v]$. Since $M \models_{X[H/v]} \psi^f$ and $Y[H/v] \subseteq X[H/v]$, we have that $M \models_{Y[H/v]} \psi^f$; and since $M \models_{Z} \psi$, this implies that $M \models_{Y[H/v]} \psi$ and that $M \models_Y \exists v \psi$. Furthermore $|Y| = |Z| \leq \height(\psi)$, as required.
\item If $\phi$ is of the form $\forall v \psi$ then, again, $\height(\phi) = \height(\psi)$. Suppose that $M \models_{X[M/v]} \psi$: again, by induction hypothesis there is a $Z \subseteq X[M/v]$ with $|Z| \leq \height(\psi)$ and such that $M \models_{Z} \psi$. For any $h \in Y$, let $\mathfrak g (h)$ pick some $s \in X$ which agrees with $h$ on all variables except $v$, and let $Y = \{\mathfrak g(h) : h \in Z\}$. Similarly to the previous case, $Z \subseteq Y[M/v] \subseteq X[M/v]$: therefore, since $M \models_{X[M/v]} \psi^f$ we have that $M \models_{Y[M/v]} \psi^f$, and since $M \models_{Z} \psi$ we have that $M \models_{Y[M/v]} \psi$. So in conclusion $M \models_Y \forall v \psi$, as required, and $|Y| = |Z| \leq n$. 
\end{enumerate}
\end{proof}
Even though constancy atoms are not upwards closed, it is possible to extend this result to $\FO(=\!\!(\cdot), \D)$. Indeed, constancy atoms are trivially $0$-bounded, since the empty team always satisfies them, and 
\begin{coro}
Let $\D$ be a family of upwards closed bounded dependencies. Then for all $\phi \in \FO(=\!\!(\cdot), \D)$
\[
	M \models_X \phi \Rightarrow \exists Y \subseteq X \mbox{ with } |Y| \leq \height(\phi) \mbox{ s.t. } M \models_Y \phi.
\]
\end{coro}
\begin{proof}
Let $\phi \in \FO(=\!\!(\cdot), \D)$: then by Lemma \ref{lem:const} $\phi$ is equivalent to some formula of the form $\exists \tuple v (=\!\!(\tuple v) \wedge \psi)$, where $\psi$ does not contain constancy atoms and $\height(\psi) = \height(\phi)$. Now suppose that $M \models_X \phi$: then, for some choice of elements $\tuple m \in \dom(M)^{|\tuple v|}$, $M \models_{X[\tuple m/\tuple v]} \psi$. Now by Theorem \ref{thm:height} there exists a $Z \subseteq X[\tuple m / \tuple v]$, with $|Z| \leq \height(\psi)$, such that $M \models_{Z} \psi$; and $Z$ is necessarily of the form $Y[\tuple m / \tuple v]$ for some $Y \subseteq X$ with $|Y| = |Z| \leq \height(\psi)$. But then $M \models_Y \exists \tuple v (=\!\!(\tuple v) \wedge \psi)$, as required.
\end{proof}
This result allows us to prove at once a number of nondefinability results concerning upwards closed dependencies. For example, it is now easy to see that 
\begin{coro}
Let $\D$ be a family of upwards closed bounded dependencies. Then the totality dependency $\all$ is not definable in $\FO(=\!\!(\cdot), \D)$. In particular, totality atoms cannot be defined by means of the negations of inclusion, exclusion, functional dependence and independence atoms.
\end{coro}
\begin{coro}
Let $\D$ be a family of $\kappa$-bounded upwards closed dependencies and let $\kappa' > \kappa$ be infinite. Then $\kappa'$-bigness is not definable in $\FO(=\!\!(\cdot), \D)$.
\end{coro}
\begin{coro}
Let $\mathbf D$ be a $k$-bounded upwards closed dependency, and let $n > k$. If $\phi(\tuple x)$ of $\FO(=\!\!(\cdot), \mathbf D)$ characterizes $n$-bigness, in the sense that for all $M$ and $X$
\[
	M \models_X \phi(\tuple x) \Leftrightarrow |X(\tuple x)| \geq n,
\]
then $\phi(\tuple x)$ contains at least $\lceil \frac{n}{k} \rceil$ instances of $\mathbf D$. 
\end{coro}
\section{Conclusions and Further Work}
In this work we discovered a surprising asymmetry between downwards closed and upwards closed first-order definable dependency conditions: whereas, as it was known since \cite{vaananen07}, the former can bring the expressive power of a logic with team semantics beyond the first order, the latter cannot do so by their own or even together with constancy atoms. As a consequence, the negations of the principal dependency notions studied so far in team semantics can all be added to first-order logic without increasing its expressive power. 

Our original question was: how much can we get away with adding to the team semantics of first-order logic before ending up in a higher order logic? The answer, it is now apparent, is \emph{quite a lot}. This demonstrates that team semantics is useful not only (as it has been employed so far) as a formalism for the study of very expressive extensions of first-order logic, but also as one for that of more treatable ones.

Much of course remains to be done. The notion of boundedness of Section \ref{sect:bounded} allowed us to find some non-definability results between our extensions; but the classification of these extensions is far from complete. In particular, it would be interesting to find necessary and sufficient conditions for $\FO(\D)$ to be equivalent to $\FO$ over sentences. The complexity-theoretic properties of these logics, or of fragments thereof, also deserve further investigation.

Another open issue concerns the development of sound and complete proof systems for our logics. Of course, one can check whether a theory $T$ implies a formula $\phi$ simply by using Theorems \ref{thm:upwards} and \ref{thm:upwards_const} to translate everything in first-order logic and then use one of the many well-understood proof systems for it; but nonetheless, it could be very informative to find out directly which logical laws our formalisms obey. 
\paragraph{Acknowledgments}
The author thanks the referees for a number of useful suggestions and corrections. 
\bibliographystyle{eptcs}
\bibliography{biblio}
\end{document}